\documentclass[12pt]{amsart}

\batchmode

\begin{document}

\def\dbl{[\hskip -1pt[}
\def\dbr{]\hskip -1pt]}
\title {Local boundedness of Catlin $q$-type  }
\author{Ozcan Yazici}
\address{ Department  of Mathematics,  Middle East  Technical University,  06800 Ankara,  Turkey }
\email{oyazici@metu.edu.tr}




\def\Label#1{\label{#1}}

\def\1#1{\ov{#1}}
\def\2#1{\widetilde{#1}}
\def\6#1{\mathcal{#1}}
\def\4#1{\mathbb{#1}}
\def\3#1{\widehat{#1}}
\def\7#1{\mathscr{#1}}
\def\K{{\4K}}
\def\LL{{\4L}}

\def \MM{{\4M}}

\def \B{{\4B}^{2N'-1}}

\def \H{{\4H}^{2l-1}}

\def \F{{\4H}^{2N'-1}}

\def \LL{{\4L}}

\def\Re{{\sf Re}\,}
\def\Im{{\sf Im}\,}
\def\id{{\sf id}\,}

\def\s{s}
\def\k{\kappa}
\def\ov{\overline}
\def\span{\text{\rm span}}
\def\ad{\text{\rm ad }}
\def\tr{\text{\rm tr}}
\def\xo {{x_0}}
\def\Rk{\text{\rm Rk\,}}
\def\sg{\sigma}
\def \emxy{E_{(M,M')}(X,Y)}
\def \semxy{\scrE_{(M,M')}(X,Y)}
\def \jkxy {J^k(X,Y)}
\def \gkxy {G^k(X,Y)}
\def \exy {E(X,Y)}
\def \sexy{\scrE(X,Y)}
\def \hn {holomorphically nondegenerate}
\def\hyp{hypersurface}
\def\prt#1{{\partial \over\partial #1}}
\def\det{{\text{\rm det}}}
\def\wob{{w\over B(z)}}
\def\co{\chi_1}
\def\po{p_0}
\def\fb {\bar f}
\def\gb {\bar g}
\def\Fb {\ov F}
\def\Gb {\ov G}
\def\Hb {\ov H}
\def\zb {\bar z}
\def\wb {\bar w}
\def \qb {\bar Q}
\def \t {\tau}
\def\z{\chi}
\def\w{\tau}
\def\Z{{\mathbb Z}}
\def\phi{\varphi}
\def\eps{\epsilon}

\def \T {\theta}
\def \Th {\Theta}
\def \L {\Lambda}
\def\b {\beta}
\def\a {\alpha}
\def\o {\omega}
\def\l {\lambda}

\def \im{\text{\rm Im }}
\def \re{\text{\rm Re }}
\def \Char{\text{\rm Char }}
\def \supp{\text{\rm supp }}
\def \codim{\text{\rm codim }}
\def \Ht{\text{\rm ht }}
\def \Dt{\text{\rm dt }}
\def \hO{\widehat{\mathcal O}}
\def \cl{\text{\rm cl }}
\def \bS{\mathbb S}
\def \bK{\mathbb K}
\def \bD{\mathbb D}
\def \bC{\mathbb C}
\def \bL{\mathbb L}
\def \bZ{\mathbb Z}
\def \bN{\mathbb N}
\def \scrF{\mathcal F}
\def \scrK{\mathcal K}
\def \mc #1 {\mathcal {#1}}
\def \scrM{\mathcal M}
\def \cR{\mathcal R}
\def \scrJ{\mathcal J}
\def \scrA{\mathcal A}
\def \scrO{\mathcal O}
\def \scrV{\mathcal V}
\def \scrL{\mathcal L}
\def \scrE{\mathcal E}
\def \hol{\text{\rm hol}}
\def \aut{\text{\rm aut}}
\def \Aut{\text{\rm Aut}}
\def \J{\text{\rm Jac}}
\def\jet#1#2{J^{#1}_{#2}}
\def\gp#1{G^{#1}}
\def\gpo{\gp {2k_0}_0}
\def\emmp {\scrF(M,p;M',p')}
\def\rk{\text{\rm rk\,}}
\def\Orb{\text{\rm Orb\,}}
\def\Exp{\text{\rm Exp\,}}
\def\Span{\text{\rm span\,}}
\def\d{\partial}
\def\D{\3J}
\def\pr{{\rm pr}}

\def \CZZ {\C \dbl Z,\zeta \dbr}
\def \D{\text{\rm Der}\,}
\def \Rk{\text{\rm Rk}\,}
\def \CR{\text{\rm CR}}
\def \ima{\text{\rm im}\,}
\def \I {\mathcal I}

\def \M {\mathcal M}

\newtheorem{Thm}{Theorem}[section]
\newtheorem{Cor}[Thm]{Corollary}
\newtheorem{Pro}[Thm]{Proposition}
\newtheorem{Lem}[Thm]{Lemma}

\theoremstyle{definition}\newtheorem{Def}[Thm]{Definition}

\theoremstyle{remark}
\newtheorem{Rem}[Thm]{Remark}
\newtheorem{Exa}[Thm]{Example}
\newtheorem{Exs}[Thm]{Examples}

\numberwithin{equation}{section}

\def\bl{\begin{Lem}}
\def\el{\end{Lem}}
\def\bp{\begin{Pro}}
\def\ep{\end{Pro}}
\def\bt{\begin{Thm}}
\def\et{\end{Thm}}
\def\bc{\begin{Cor}}
\def\ec{\end{Cor}}
\def\bd{\begin{Def}}
\def\ed{\end{Def}}
\def\be{\begin{Exa}}
\def\ee{\end{Exa}}
\def\bpf{\begin{proof}}
\def\epf{\end{proof}}
\def\ben{\begin{enumerate}}
\def\een{\end{enumerate}}

\newcommand{\dbar}{\bar\partial}
\newcommand{\genmat}{\lambda}
\newcommand{\polynorm}[1]{{|| #1 ||}}
\newcommand{\vnorm}[1]{\left\|  #1 \right\|}
\newcommand{\asspol}[1]{{\mathbf{#1}}}
\newcommand{\Cn}{\mathbb{C}^n}
\newcommand{\Cd}{\mathbb{C}^d}
\newcommand{\Cm}{\mathbb{C}^m}
\newcommand{\C}{\mathbb{C}}
\newcommand{\CN}{\mathbb{C}^N}
\newcommand{\CNp}{\mathbb{C}^{N^\prime}}
\newcommand{\Rd}{\mathbb{R}^d}
\newcommand{\Rn}{\mathbb{R}^n}
\newcommand{\RN}{\mathbb{R}^N}
\newcommand{\R}{\mathbb{R}}
\newcommand{\bR}{\mathbb{R}}
\newcommand{\N}{\mathbb{N}}
\newcommand{\dop}[1]{\frac{\partial}{\partial #1}}
\newcommand{\vardop}[3]{\frac{\partial^{|#3|} #1}{\partial {#2}^{#3}}}
\newcommand{\br}[1]{\langle#1 \rangle}
\newcommand{\infnorm}[1]{{\left\| #1 \right\|}_{\infty}}
\newcommand{\onenorm}[1]{{\left\| #1 \right\|}_{1}}
\newcommand{\deltanorm}[1]{{\left\| #1 \right\|}_{\Delta}}
\newcommand{\omeganorm}[1]{{\left\| #1 \right\|}_{\Omega}}
\newcommand{\nequiv}{{\equiv \!\!\!\!\!\!  / \,\,}}
\newcommand{\bk}{\mathbf{K}}
\newcommand{\p}{\prime}
\newcommand{\tV}{\mathcal{V}}
\newcommand{\poly}{\mathcal{P}}
\newcommand{\ring}{\mathcal{A}}
\newcommand{\ringk}{\ring_k}
\newcommand{\ringktwo}{\mathcal{B}_\mu}
\newcommand{\germs}{\mathcal{O}}
\newcommand{\On}{\germs_n}
\newcommand{\mcl}{\mathcal{C}}
\newcommand{\formals}{\mathcal{F}}
\newcommand{\Fn}{\formals_n}
\newcommand{\autM}{{\Aut (M,0)}}
\newcommand{\autMp}{{\Aut (M,p)}}
\newcommand{\holmaps}{\mathcal{H}}
\newcommand{\biholmaps}{\mathcal{B}}
\newcommand{\autmaps}{\mathcal{A}(\CN,0)}
\newcommand{\jetsp}[2]{ G_{#1}^{#2} }
\newcommand{\njetsp}[2]{J_{#1}^{#2} }
\newcommand{\jetm}[2]{ j_{#1}^{#2} }
\newcommand{\glnc}{\mathsf{GL_n}(\C)}
\newcommand{\glmc}{\mathsf{GL_m}(\C)}
\newcommand{\glc}{\mathsf{GL_{(m+1)n}}(\C)}
\newcommand{\glk}{\mathsf{GL_{k}}(\C)}
\newcommand{\smC}{\mathcal{C}^{\infty}}
\newcommand{\anC}{\mathcal{C}^{\omega}}
\newcommand{\kC}{\mathcal{C}^{k}}


 
\begin{abstract} 

In \cite{D}, D'Angelo introduced the notion of  finite type for points $p$ of a real hypersurface $M$ of $\mathbb C^n$ by defining the order of contact $\Delta_q(M,p)$ of complex analytic $q$-dimensional varieties with $M$ at $p$.  Later, Catlin \cite{C} defined $q$-type, $D_q(M,p)$ for points of hypersurfaces by considering  generic $(n-q+1)$-dimensional complex affine  subspaces of $\mathbb C^n$. We define a generalization of the Catlin's $q$-type for an arbitrary subset $M$ of $\mathbb C^n$ in a similar way that D'Angelo's 1-type, $\Delta_1(M,p)$, is generalized  in \cite{LM}. Using recent results connecting the D'Angelo and Catlin $q$-types in \cite{BN1} and building on D'Angelo's work on the openness of the set of points of finite $\Delta_q$-type, we prove the openness of the set of points of finite Catlin $q$-type for an arbitrary subset $M\subset \mathbb C^n$. \\

\end{abstract}

\maketitle

\section{Introduction}\Label{int} 

Let $\Omega$ be a pseudoconvex domain in $\mathbb C^n$ with real-analytic boundary. Subellipticity of $\bar \partial-$Neumann problem at a boundary point $p$ depends on the order of contact of complex analytic varieties with $\partial \Omega$ at $p$. In \cite{k}, Kohn proved that subellipticity of $\bar \partial-$Neumann problem at a boundary point $p$ on $(0,q)$ forms is equivalent to non-existence of $q$-dimensional complex-analytic varieties in $\partial \Omega$ through $p$. To measure the order of contact of holomorphic curves with  a smooth real hypersurface  $M\subset\mathbb C^n$ at $p$,  D'Angelo  \cite{D}  introduced the notion of  type.  More precisely,  the type of $M$ at $p$ is defined by $$\Delta(M,p)=\sup_{\gamma\in\mathcal C} \frac{\nu(r\circ \gamma)}{\nu(\gamma)}$$ where $r$ is a defining function of $M$, $\mathcal C$ is the set of non-constant holomorphic germs of curves $\gamma$ at $0\in \mathbb C$ so that  $\gamma(0)=p$ and $\nu(r\circ\gamma)$ denotes the order of vanishing of the function $r\circ\gamma$ at $0.$ A point $p$ is called finite type if $\Delta(M,p)<\infty$.  In \cite{D}, D'Angelo proved the crucial property that the set of  points of finite type forms an open subset of $M$. This condition of finite type appeared later to be central in Catlin's work \cite{C} on subelliptic estimates for the $\bar \partial$-Neumann problem. See \cite{JM,F2,KZ} for a more recent discussion of the relationship between finite type and subellipticity.

In  \cite{D}, D'Angelo defined the $q$-type of a hypersurface $M\subset \mathbb C^n$, which possibly contains $q-1$ dimensional complex analytic varieties.  The $q$-type of $M$ at $p\in M$ is defined by $$\Delta_q(M,p)=\inf\{\Delta(M\cap P,p): P\;\text{is any }\; n-q+1 \;\text{dimensional complex affine subspace of }\; \mathbb C^n  \}.$$ In this definition, $M\cap P$ is considered as the germ at $p$ of a smooth real hypersurfaces in $\mathbb C^{n-q+1}.$ We should note that, when $q=1$, $\Delta_1(M,p)=\Delta(M,p).$ 

Let $\mathcal O_p$ be the ring of  germs of holomorphic functions  at $p$ in $\mathbb C^n$ and $\mathcal C^{\infty}_p$ be the  ring of  germs of smooth functions at $p$. $\Delta$ can also be defined for ideals in $\mathcal O_p$ or $\mathcal C^{\infty}_p$. For a proper ideal $I$ in $\mathcal O_p$ or in $\mathcal C^{\infty}_p$, in \cite{D,D2}, D'Angelo defined  \begin{eqnarray*}   \Delta(I,p)=\sup_{\gamma\in\mathcal C}\inf_{\phi\in I}\frac{\nu(\phi\circ\gamma)}{\nu(\gamma)}\\ \Delta_q(I,p)=\inf_{w}\Delta((I,w),p),
 \end{eqnarray*}  
where infimum in the definition of $\Delta_q$ is taken over  $q-1$ linear forms $w=\{w_1,\dots,w_{q-1}\}$ in $\mathcal O_p$ and $(I,w)$ is the ideal generated by $I$ and $w$. 

If an ideal $I$ in $\mathcal O_p$ contains $q$ independent linear functions, then it follows from Theorem 2.7 in \cite{D} that $$\Delta(I,p)\leq mult(I,p)\leq (\Delta(I,p))^{n-q},$$ where 
$mult(I)=\dim_{\mathbb C}(\mathcal O_p/I). $

D'Angelo $q$-type for a real hypersurface $M$ of $\mathbb C^n$ has an equivalent definition (see p.86 of \cite{D2}): $$\Delta_q(M,p)=\Delta_q(I(M),p)$$ where  $I(M)$ is the ideal of smooth germs in $\mathcal C^{\infty}_p$,  vanishing on $M$ near $p$.

In \cite{C}, Catlin defined $q$-type, $D_q(M,p)$,  of a  smooth real hypersurface $M\subset\mathbb C^n$ at $p$ by considering the intersections of germs at $p$ of $q$ dimensional varieties with generic $(n-q+1)$-dimensional complex affine  subspaces of $\mathbb C^n$ through $p$. More precisely, Catlin $q-$type is defined by $$D_q(M,p)=\sup_{V^q} \underset{S\in G_p^{n-q+1}}{\text{gen.val}} \max_{k=1,\dots,P}\frac{\nu(r\circ\gamma^k_S)}{\nu(\gamma_S^k)}. $$ Here the supremum is taken over the germs at $p$ of $q$-dimensional complex varieties $V^q$, $G_p^{n-q+1}$ denotes the set of $(n-q+1)$-dimensional complex affine subspaces of $\mathbb C^n$ through $p$ and $\gamma_S^k $ denote the germs of the one-dimensional irreducible components of $V^q\cap S$. In \cite[Proposition 3.1]{C}, Catlin showed that for any germ at $p$ of a $q$-dimensional complex variety $V^q$, there exists an open (hence, dense) subset $W$ of $G^{n-q+1}_p$ such that for any $S\in W$, $V^q\cap S$ has same number of one-dimensional components and $$\max_{k=1,\dots,P}\frac{\nu(r\circ\gamma^k_S)}{\nu(\gamma_S^k)}$$ is independent of $S\in W$. Hence, the  generic value in the definition of $D_q(M,p)$ is computed by any $S\in W$.
Similarly, Catlin also defined $D_q$ type for an ideal $I$ in $\mathcal O_p$ by 

$$ \displaystyle D_q(I,p)=\sup_{V^q}\inf_{\phi\in I} \underset{S\in G_p^{n-q+1}}{\text{gen.val}} \max_{k=1,\dots,P}\frac{\nu(\phi\circ\gamma^k_S)}{\nu(\gamma_S^k)}.$$

In a combination of work  \cite{C83,C84,C}, Catlin showed that finite $D_q$-type  at point $p$ of smooth real hypersurface $M$ is equivalent to subellipticity of the $\bar \partial$-Neumann problem for $(0,q)$ forms  at $p$.  Since subellipticity is an open condition, this result implies in particular that, for a smooth real hypersurface $M$, the set of points of finite $D_q$-type is an open subset of $M$.

When $q=1$, $$D_1(M,p)=\Delta_1(M,p)=\Delta(M,p).$$  For a long time, Catlin  $q$-type, $D_q(M,p)$, and D'Angelo $q$-type,  $\Delta_q(M,p)$, were believed to be equal.  In \cite{F}, Fassina gave  examples of ideals and hypersurfaces to show that these two types can be different when $q\geq2.$ 

For a smooth real hypersurface $M$ in $\mathbb C^n$, Brinzanescu and Nicoara  \cite{BN1} introduced another definition of $q$-type at a point $p\in M$ by    $$\tilde\Delta_q(M,p)=\underset{w}{\text{gen.val}} \, \Delta((I(M),w),p),$$
where the generic value is taken over all $q-1$ linear forms $w=\{w_1,\dots, w_{q-1}\}$ in $\mathcal O_p$ and $(I(M),w)$ is the ideal generated by $I(M)$ and $w$ in $\mathcal C^{\infty}_p$. A similar definition is given for a proper ideal $I$ in $\mathcal O_p$ by   \begin{eqnarray*}    \tilde \Delta_q(I,p)=\underset{w}{\text{gen.val}}\,\Delta((I,w),p).
 \end{eqnarray*}

In the same paper, they showed that \begin{eqnarray}\label{1.1} D_q(M,p)&=&\tilde\Delta_q(M,p),\\\nonumber D_q(I,p)&=&\tilde\Delta_q(I,p). \end{eqnarray}

Using these equalities, a relation is given between $D_q(I,p)$ and $\Delta_q(I,p)$ in \cite{F} and \cite{BN1}, which implies that $\Delta_q(I,p)$ and $D_q(I,p)$ are simultaneously finite.

Recently, in their study of the ${ C}^\infty$ regularity problem for CR maps between smooth CR manifolds, Lamel and Mir \cite{LM} considered finite D'Angelo-type points for arbitrary subsets $M\subset \mathbb C^n$. In \cite{Y}, the author showed that, for any subset $M$ of $\mathbb C^n$, the set of finite D'Angelo-type points is open in $M.$  

In a similar way that D'Angelo type, $\Delta(M,p)$, is generalized  in \cite{LM}, we define a generalization of the Catlin's $q$-type for an arbitrary subset $M \subset\mathbb C^n$ by
 $$D_q(M,p)=\sup_{V^q} \inf_{r\in I(M)} \underset{S\in G_p^{n-q+1}}{\text{gen.val}} \max_{k=1,\dots,P}\frac{\nu(r\circ\gamma^k_S)}{\nu(\gamma_S^k)}. $$

In this note, by combining D'Angelo's arguments with the ideas in the proof the equality (\ref {1.1}), we establish the local boundedness of Catlin $q$-type for an arbitrary subset of $\mathbb C^n$. More precisely, our main result is the following.

\begin{Thm}\label{main}Let $M$ be a subset  of $\mathbb C^n$ and $p_0$ be a point of finite Catlin $q$-type, that is, $D_q(M,p_0)<\infty.$ Then there is a neighborhood $V$ of $p_0$ so that $$D_q(M,p)\leq 2^{(n-q+1)^2+n-q+2}D_q(M,p_0)^{(n-q+1)^2}$$ for all $p\in V$. In particular, the set of points of finite Catlin $q$-type is an open subset of $M$.  
\end{Thm}

In \cite{D1}, D'Angelo gave an example of hypersurface on which $\Delta(M,p)$ fails to be upper-semi continuous. In a similar way, the following example shows that $D_2(M,p)$ is not upper-semi continuous. 
\begin{Exa}Let $M$ be the real hypersurface in $\mathbb C^4$ with the defining function  $$r(z)=\Re z_4+|z_1^2-z_2z_3|^2+|z_2|^4. $$ Let $w=az_1+bz_2+cz_3+z_4$ be a generic linear form where $a,b,c$ are all non-zero and consider the curve $\gamma_1(t)=(t,\frac{-at}{b},0,0)\subset\{w=0\}$. Then $\nu(r\circ\gamma_1)=4$ and hence $D_2(M,0)=\tilde\Delta_2(M,0)=4$.\\

Let $p=(0,0,\epsilon,0)$ be a point on $M$ and $w=az_1+bz_2+z_3+cz_4-\epsilon$ be a linear form. By solving $w=0$, $z_1^2-z_2z_3$ and $z_4=0$ together, we obtain that $$z_2=\frac{(\epsilon-az_1)\pm(\epsilon-az_1)\sqrt{1-\frac{4bz_1^2}{(\epsilon-az_1)^2}}}{2b}. $$ In local parametrization, we write the curve $\gamma_2$ as $$\gamma_2(t)=(t,\frac{t^2}{\epsilon-at}+O(t^4), \epsilon-at-\frac{bt^2}{\epsilon-at}+O(t^4),0).$$ Here we choose the root for $z_2$ with the minus sign.  Then $\nu(r\circ\gamma_2)=8$, which implies that $D_2(M,p)=\tilde\Delta_2(M,p)=8$. Hence $D_2(M,p)$ is not an upper-semi continuous function of $p$.
\end{Exa}

\section{Proof of the Main Result}\Label{int} 

Let $M$ be a subset of $\mathbb C^n$. In this section, we will show that $D_q(M,p)$ is locally bounded by above which, in particular, implies that the set of points of finite Catlin $q$-type is an open subset of $M$. We need give some more definitions analogous to the ones in \cite{D}.  \\

 For any  $r\in I(M)$ and $k\in\mathbb Z^+$, we denote by $r_k$ the Taylor polynomial of $r$ of order $k$ at $p$. $I(M_k)$ denotes the ideal generated by the set $\{r_k: r\in I(M) \}$ in $\mathcal C^{\infty}_p$.   

 $\Delta_q(M_k,p)$ is defined by  $$\Delta_q(M_k,p)=\Delta_q(I(M_k),p). $$ Similarly,  $D_q(M_k,p)$ is defined by

$$D_q(M_k,p)=\sup_{V^q} \inf_{r\in I(M_k)} \underset{S\in G_p^{n-q+1}}{\text{gen.val}} \max_{j=1,\dots,P}\frac{\nu(r\circ\gamma^j_S)}{\nu(\gamma_S^j)}. $$

By Proposition 3.1 in \cite{D}, we can decompose the polynomial $r_k$ as $$r_k=\Re (h^k)+\sum_{j=1}^{N}| f_j^k |^2-\sum_{j=1}^{N}|g_j^k |^2,$$ where  $N=N_k$ depends only on $n,k$,   $h$ is a holomorphic function, $(f_k)_{1}^{N}$ and  $(g_k)_{1}^{N}$ are holomorphic mappings. Let $\mathcal U(N_k)$ denote the group of unitary matrices on $\mathbb C^{N_k}$. For any $U\in\mathcal U(N_k)$, we denote by $I(r,U,k,p)$, the ideal generated by $h^k$ and $f^k-Ug^k$.
We should note that the decomposition of $r_k$  is not unique and $I(r,U,k,p)$ depends on the choice of decomposition. Here, we use the one in the proof of Proposition 3.1 in \cite{D}. By $I(U,k,p)$, we denote the ideal generated by  functions  $h^k$ and $f^k-Ug^k$ corresponding to the decomposition of $r_k$ for all $r\in I(M)$.

\begin{Lem}\label{L3}Let $M$ be a subset of $\mathbb C^n$.  If $D_q(M_k,p)<k$  for some $k\in \mathbb Z^+$, then $D_q(M_k,p)=D_q(M,p)$. 
\end{Lem}
\begin{proof} For any germ at $p$ of $q$-dimensional variety $V^q$ and small enough $\epsilon>0$,  
there exists an $r^{0}\in I(M)$ such that
 \[
    \underset{S\in G_p^{n-q+1}}{\text{gen.val}}       \max_{j=1,...,P}\frac{\nu(r^0_k\circ\gamma_S^j)}{\nu(\gamma_S^j)}<D_q(M_k,p)+\epsilon<k.
  \]
Let  $W$ and $W_k$ be open dense  subsets of $G^{n-q+1}$ in which the generic values are assumed in the definitions of $D_q(M,p)$, $D_q(M_k,p)$ for $V^q$, respectively. Then for all $S\in W\cap W_k$, 
  \[ \max_{j=1,...,P}\frac{\nu(r^0_k\circ\gamma_S^j)}{\nu(\gamma_S^j)}<D_q(M_k,p)+\epsilon<k
  \] where $\gamma_S^j$'s  are connected components of $V^q\cap S$.   
   As in the proof of Lemma 4.5 in  \cite{D},  we have 
\[
\frac{\nu(r^0\circ \gamma_S^j)}{\nu(\gamma_S^j)}=\frac{\nu(r^0_k\circ \gamma_S^j)}{\nu(\gamma_S^j)},
\]
for all $j=1,...,P.$ Thus \[
    \underset{S\in G_p^{n-q+1}}{\text{gen.val}}       \max_{j=1,...,P}\frac{\nu(r^0\circ\gamma_S^j)}{\nu(\gamma_S^j)}=\max_{j=1,...,P}\frac{\nu(r^0\circ\gamma_S^j)}{\nu(\gamma_S^j)}=\max_{j=1,...,P}\frac{\nu(r^0_k\circ\gamma_S^j)}{\nu(\gamma_S^j)}<D_q(M_k,p)+\epsilon,
  \]
which implies  that $$D_q(M,p)\leq D_q(M_k,p)+\epsilon.$$ As $\epsilon >0$ is arbitrary, $D_q(M,p)\leq D_q(M_k,p).$

Since $D_q(M,p)\leq D_q(M_k,p)<k$, following the above argument, for any germ at $p$ of $q$-dimensional variety $V^q$ and small enough $\epsilon>0$,  
there exists an $r'\in I(M)$ such that
 \[
\underset{S\in G_p^{n-q+1}}{\text{gen.val}}       \max_{j=1,...,P}\frac{\nu(r'\circ\gamma_S^j)}{\nu(\gamma_S^j)}<D_q(M,p)+\epsilon<k.
  \]
Let $W$ and $W_k$ be open dense  subsets of $G^{n-q+1}$ in which the generic values are assumed in the definitions of $D_q(M,p)$  and  $D_q(M_k,p)$ for $V^q$, respectively. 
Then by the similar argument above, for all $S\in W\cap W_k$, 
   \[
    \underset{S\in G_p^{n-q+1}}{\text{gen.val}}       \max_{j=1,...,P}\frac{\nu(r'_k\circ\gamma_S^j)}{\nu(\gamma_S^j)}=\max_{j=1,...,P}\frac{\nu(r'_k\circ\gamma_S^j)}{\nu(\gamma_S^j)}=\max_{j=1,...,P}\frac{\nu(r'\circ\gamma_S^j)}{\nu(\gamma_S^j)}<D_q(M,p)+\epsilon,
  \]
which implies  that $$D_q(M_k,p)\leq D_q(M,p).$$

\end{proof}

\begin{Lem} \label{L2} Let $M\subset \mathbb{C}^n$ be a subset.  Then $$\sup_U\Delta_q(I(U,k,p)) \leq D_q(M_k,p)\leq 2\sup_U(\Delta_q(I(U,k,p)))^{n-q+1}.$$
\end{Lem}
\begin{proof}For any $r\in I(M)$ and $k\in \mathbb Z^+$, as in the proof of Theorem 3.4 in \cite{JM}, we have that $$\inf_{\phi\in I(r,U,k,p)}\nu(\phi\circ \gamma)\leq \nu(r_k \circ\gamma)$$ for any curve $\gamma\in \mathcal C$.  Thus for any set of $q-1$ linear forms $w$, $$ \inf_{\phi\in (I(r,U,k,p),w)}\frac{\nu(\phi\circ \gamma)}{\nu(\gamma)}\leq\inf_{\phi\in(r_k,w)} \frac{\nu(\phi \circ\gamma)}{\nu(\gamma)},$$
where $(r_k,w)$ is the ideal generated by $r_k$ and $w$ in $\mathcal C^{\infty}_p$.

By taking infimum over $r$, supremum over $\gamma$ and infimum over $w$, we obtain that 

\begin{eqnarray}\label{ex} \Delta_q(I(U,k,p))\leq \inf_w\sup_{\gamma}\inf_{\phi\in (I(M_k),w)}\frac{\nu(\phi\circ\gamma)}{\nu(\gamma)}.
\end{eqnarray}
Let $w_0$ be a set of $q-1$ linear forms at which  right hand side of (\ref{ex})  attains its infimum.  If a curve $\gamma$ is not contained in the set  $\{w_0=0\}$ then there exists a linear function $w_0^j$ so that $\nu(w_0^j\circ\gamma)=\nu(\gamma).$ Thus the supremum above is attained for the curves which are contained in $\{w_0=0\}$. For any such curve, infimum is attained for $\phi\in I(M_k)$. Hence, there exists a $\gamma_0\subset \{w_0=0\}$ such that 
\begin{eqnarray} \label{ex2}\Delta_q(I(U,k,p))\leq \inf_{r\in I(M_k)}\frac{\nu(r\circ \gamma_0)}{\nu(\gamma_0)}.
\end{eqnarray}

Let $H$ be denote the linear subspace $\{w_0=0\}$ in $G^{n-q+1}_p$ and $Z$ be the $q-1$ dimensional hyperplane through $p$, that is transversal to $H$. That is, $\dim _{\mathbb C}(H\bigoplus Z)=n$.  By Lemma 3.1 in \cite{BN1}, there exists a $q$-dimensional variety $C_Z^q$ at $p$ that contains $\gamma_0$ and whose tangent spaces at $p$ contains $Z$.  As in the proof of Proposition 3.5 in \cite{BN1}, $H\in G_p^{n-q+1}$ is one of the hyperplane which gives the generic value in $D_q(M_k,p)$ for $C_Z^q$. Thus $\gamma_0\subset H\cap C_Z^q$ is one of the curves which enter in the computation of $D_q(M_k,p)$. That is,
$$\inf_{r\in I(M_k)}\frac{\nu(r\circ\gamma_0)}{\nu(\gamma_0)}\leq D_q(M_k,p) .$$ Hence by (\ref{ex2}), $$\Delta_q(I(U,k,p))\leq D_q(M_k,p).$$ By taking supremum over $U\in \mathcal U(N_k)$, the inequality on the left hand side in Lemma \ref{L2} follows.

There exists a germ at $p$ of $q$-dimensional variety $V^q$ and an open subset $W$ of $G^{n-q+1}_p$, which depends on $V^q$, such that  for all $S\in W$ 

\begin{eqnarray} \label{ex4} D_q(M_k,p)=\inf_{r\in I(M_k)}\max_{j=1,...,P}\frac{\nu(r\circ\gamma_S^j)}{\nu(\gamma_S^j)} \leq \sup_{\gamma\subset \{ w=0\}}\inf_{r\in I(M_k)}\frac{\nu(r\circ\gamma)}{\nu(\gamma)}
\end{eqnarray}

where $w$ is the set of $q-1$ linear forms with   $\{w=0\}=S$ and   $\gamma_s^j$'s are the connected components of $V^q\cap S$. 
 
By Theorem 3.5 in \cite{D}, for any curve $\gamma$ and  $r\in I(M)$, there exists a unitary matrix  $U\in\mathcal U(N_k)$ such that $\nu(r_k\circ \gamma)\leq 2\nu(\phi\circ \gamma)$ for all $\phi\in I(r,U,k,p)$.  Dividing both sides by $\nu(\gamma)$, taking infimum over $r\in I(M)$, supremum over $\gamma\subset \{w=0\}$, we obtain by (\ref{ex4}) that 
 \begin{eqnarray*} D_q(M_k,p)&\leq& 2\sup_{U}\sup_{\gamma\subset\{w=0\}}\inf_{\phi\in(I(U,k,p))}\frac{\nu(\phi\circ \gamma)}{\nu(\gamma)}= 2\sup_{U}\sup_{\gamma}\inf_{\phi\in(I(U,k,p),w)}\frac{\nu(\phi\circ \gamma)}{\nu(\gamma)}\\&=&2\sup_U\Delta_1(I(U,k,p),w)  \leq 2\sup_{U} \;mult(I(U,k,p),w)\leq 2\sup_U\; mult(I(U,k,p),\tilde w_U)\\&\leq& 2\sup_U(\Delta_1(I(U,k,p),\tilde w_U))^{n-q+1}=2\sup_U(\Delta_q(I(U,k,p)))^{n-q+1} ,
\end{eqnarray*} 
where $\tilde w_U$ is the set of $q-1$ linear functions at which $\Delta_q(I(U,k,p)$ assumes its infimum. Second and fourth inequalities above follow from Theorem 2.7 in \cite{D}.  Third inequality holds since $mult(I)$ is an upper semi continuous function of the generators of $I$. \\

\end{proof}

\begin{Lem}\label{L1} If $\Delta_q(M_k,p)<k$ then  $\Delta_q(M,p)=\Delta_q(M_k,p)$.
\end{Lem}

\begin{proof}  Let $\displaystyle \Delta_q(M_k,p)=\sup_{\gamma}\inf_{\phi\in(I(M_k),w_0)}\frac{\nu(\phi\circ\gamma)}{\nu(\gamma)}$ for some set of $q-1$  linear forms $w_0$. As in the proof of Lemma \ref{L2}, the supremum above is attained for the curves which are contained in $\{w_0=0\}$. For any such curve $\gamma$,  
$\displaystyle \inf_{\phi\in(I(M_k),w_0)}\frac{\nu(\phi\circ\gamma)}{\nu(\gamma)}$ is attained for $\phi\in I(M_k)$. 
Hence $$\Delta_q(M_k,p)=\sup_{\gamma\subset\{w_0=0\} }\inf_{r\in I(M)}\frac{\nu(r_k\circ \gamma)}{\nu(\gamma)}<k.$$

For all $\gamma\subset \{w_0=0 \}$, $\epsilon>0$ small enough, there exists $r^0\in I(M)$ such that $$\frac{\nu(r^0_k\circ\gamma)}{\nu(\gamma)}< \Delta_q(M_k,p)+\epsilon<k.$$ 
By Lemma 4.5 in \cite{D}, $\frac{\nu(r^0_k\circ\gamma)}{\nu(\gamma)}=\frac{\nu(r^0\circ\gamma)}{\nu(\gamma)}$ for any curve $\gamma.$ By taking infimum over $r\in I(M)$, supremum over $\gamma\subset \{w_0=0\}$ in $\frac{\nu(r\circ \gamma)}{\nu(\gamma)}$, we obtain that $\Delta_q(M,p)\leq \Delta_q(M_k,p)+\epsilon.$ 

By the same argument above, $$\Delta_q(M,p)=\inf_w\sup_{\gamma\subset\{w=0\}} \inf_{r\in I(M)} \frac{\nu(r\circ \gamma)}{\nu(\gamma)}<k. $$ Let the infimum be attained for some set of $q-1$ linear forms, $w_1$. Then for all $\gamma\subset\{w_1=0 \}$, there exists $r^0\in I(M)$ such that $\frac{\nu(r^0\circ \gamma)}{\nu(\gamma)}= \inf_{r\in I(M)} \frac{\nu(r\circ \gamma)}{\nu(\gamma)}<k.$ Lemma 4.5 in \cite{D} implies that $\frac{\nu(r^0_k\circ \gamma)}{\nu(\gamma)}=\frac{\nu(r^0\circ \gamma)}{\nu(\gamma)}$.  By taking infimum over $r\in I(M)$ and supremum over $\gamma\subset \{w_1=0\}$ in  $\frac{\nu(r_k\circ \gamma)}{\nu(\gamma)}$, we obtain that
 $$\Delta_q(M_k,p)\leq \sup_{\gamma\in \{w_1=0\}}\inf_{r\in I(M)}\frac{\nu(r_k\circ\gamma)}{\nu(\gamma)}=\Delta_q(M,p). $$ 
\end{proof}

\begin{Lem}\label{L4} $\sup_{U}\Delta_q(I(U,k,p))\leq \Delta_q(M_k,p)\leq 2\sup_{U}\Delta_q(I(U,k,p)).$
\end{Lem}

\begin{proof}  Let $w_0$ be a set of linear functions such that $$\Delta_q(M_k,p)=\sup_{\gamma}\inf_{\phi\in(I(M_k),w_0)}\frac{\nu(\phi\circ\gamma)}{\nu(\gamma)}.$$  By the same argument in the proof of Lemma \ref{L1}, $$\Delta_q(M_k,p)=\sup_{\gamma\subset \{w_0=0\}}\inf_{r\in I(M)} \frac{\nu(r_k\circ\gamma)}{\nu(\gamma)}.$$  By the proof of Theorem 3.4 in \cite{JM}, for any $r\in I(M)$ and any curve $\gamma$,  $$\inf_{\phi\in I(r,U,k,p)}\nu(\phi\circ \gamma)\leq \nu(r_k\circ\gamma),$$ for all $U\in \mathcal U(N_k)$. 
By taking infimum over $r\in I(M)$, we get $$\inf_{\phi\in I(U,k,p)}\frac{\nu(\phi\circ\gamma)}{\nu(\gamma)}\leq \inf_{r\in I(M)}\frac{\nu(r_k\circ\gamma)}{\nu(\gamma)}. $$
Thus

\begin{eqnarray*} \Delta_q(I(U,k,p))&=&\inf_w\sup_{\gamma\subset \{w=0\}}\inf_{\phi\in I(U,k,p)}\frac{\nu(\phi\circ\gamma)}{\nu(\gamma)}\leq\sup_{\gamma\subset \{w_0=0\}}\inf_{\phi\in I(U,k,p)}\frac{\nu(\phi\circ\gamma)}{\nu(\gamma)}\\&\leq& \sup_{\gamma\subset \{w_0=0\}}\inf_{r\in I(M)}\frac{\nu(r_k\circ\gamma)}{\nu(\gamma)}=\Delta_q(M_k,p).\end{eqnarray*}
Hence $\sup_{U}\Delta_q(I(U,k,p))\leq \Delta_q(M_k,p).$\\

For any set of $q-1$ linear forms $w$ and any curve $\gamma\subset \{w=0\}$, by Theorem 3.5 in \cite{D}, there exists a unitary matrix $U\in \mathcal U(N_k)$, such that $ \frac{\nu(r_k\circ\gamma)}{\nu(\gamma)}\leq 2\inf_{\phi\in I(r,U,k,p)}\frac{\nu(\phi\circ\gamma)}{\nu(\gamma)}$ for all  $r\in I(M)$. By taking infimum of $\frac{\nu(r_k\circ \gamma)}{\nu(\gamma)}$, over $r\in I(M)$, supremum over $\gamma\subset\{w=0\}$ and infimum over the set of $q-1$ linear forms $w$, we obtain that $$ \Delta_q(M_k,p)\leq 2\sup_U\inf_{w}\sup_{\gamma\subset\{w=0\}}\inf_{\phi\in I(U,k,p)}\frac{\nu(\phi\circ\gamma)}{\nu(\gamma)}=2\sup_{U}\Delta_q(I(U,k,p)).$$
\end{proof}

Now we can prove our main theroem.

\begin{proof}[Proof of Theorem \ref{main}]
 Let $w_0$ be the set of $q-1$ linear forms such that   $\Delta_q(M_k, p_0)$ attains its infimum. Then   $$\displaystyle \Delta_q(M_k, p_0)=\sup_{\gamma\subset\{w_0=0\} }\inf_{r\in I(M)}\frac{\nu(r_k\circ\gamma)}{\nu(\gamma)}.$$ By identifying $\{w_0=0\}=\mathbb C^{n-q+1}$, Theorem 1.3 in \cite{Y} implies that  there is a neighborhood $V$ of $p_0$ such that for all $p\in V$

$$\Delta_q(M_k,p)\leq 2(\Delta_q(M_k,p_0))^{n-q+1}.$$ For all $p\in V$, 
\begin{eqnarray*}
D_q(M_k,p)&\leq& 2\sup_U\Delta_q(I(U,k,p))^{n-q+1}\leq 2\Delta_q(M_k,p)^{n-q+1}\\&\leq& 2^{n-q+2}\Delta_q(M_k,p_0)^{(n-q+1)^2}\leq 2^{n-q+2}(2\sup_U\Delta_q(I(U,k,p_0)))^{(n-q+1)^2}\\&\leq& 2^{(n-q+1)^2+n-q+2}D_q(M_k,p_0)^{(n-q+1)^2}=2^{(n-q+1)^2+n-q+2}D_q(M,p_0)^{(n-q+1)^2}.
\end{eqnarray*}

First and the last inequalities above follow from Lemma \ref{L2}. Second and fourth inequalities follow from Lemma \ref{L4}. Last equality holds for  large enough $k$ by Lemma \ref{L3}. If $k> 2^{(n-q+1)^2+n-q+2}D_q(M,p_0)^{(n-q+1)^2}$, Lemma \ref{L3} implies that $D_q(M_k,p)=D_q(M,p)$  hence the theorem follows.

\end{proof}

\noindent \textbf{Acknowledgments.} I am grateful to Prof. J.P. D'Angelo for his suggestion to work on $q$-types. I would like to thank to the referee for his/her important comments, corrections and suggestions which improved the presentation of the paper significantly. The author is supported by T\"UB\.ITAK 3501 Proj. No 120F084.

\end{document}